\documentclass[12pt]{amsart}
\usepackage[margin=0.9in]{geometry}
\usepackage[utf8]{inputenc}
\usepackage[T1]{fontenc}
\usepackage{lmodern, appendix, comment}
\usepackage{amsmath,amssymb,amsfonts,mathtools}
\usepackage{amsthm}
\usepackage{enumitem}
\usepackage{microtype}
\usepackage{xcolor}
\usepackage[colorlinks=true,linkcolor=blue!50!black,citecolor=blue!50!black,urlcolor=blue!50!black]{hyperref}

\DeclareUnicodeCharacter{2013}{--}
\DeclareUnicodeCharacter{2014}{---}
\DeclareUnicodeCharacter{2019}{'}
\DeclareUnicodeCharacter{00E4}{\"a}
\DeclareUnicodeCharacter{00E9}{\'e}

\numberwithin{equation}{section}
\newtheorem{theorem}{Theorem}[section]
\newtheorem{proposition}[theorem]{Proposition}
\newtheorem{lemma}[theorem]{Lemma}
\newtheorem{corollary}[theorem]{Corollary}

\theoremstyle{definition}

\theoremstyle{remark}

\newtheorem{remarks}[theorem]{Remarks}
\newtheorem{conjecture}[theorem]{Conjecture}

\newcommand{\C}{\mathbb C}

\newcommand{\Bball}{\mathbb B}

\DeclareMathOperator{\Aut}{Aut}

\DeclareMathOperator{\U}{U}
\DeclareMathOperator{\tr}{tr}
\DeclareMathOperator{\Spec}{Spec}
\DeclareMathOperator{\Reg}{Reg}

\newcommand{\trans}{\mathsf T}

\title[On the Classification of Stein spaces with Bergman--Einstein metrics]{On the Classification of Stein Spaces with Bergman--Einstein Metrics}

\author{Soumya Ganguly}
\address{Department of Mathematics, Rutgers University--New Brunswick, 110 Frelinghuysen Rd, Piscataway, NJ 08854, USA}
\email{soumya.ganguly@rutgers.edu}

\author{Siddhartha Sahi}
\address{Department of Mathematics, Rutgers University--New Brunswick, 110 Frelinghuysen Rd, Piscataway, NJ 08854, USA}
\email{sahi@math.rutgers.edu}


\begin{document}

\begin{abstract}
For every \(N\ge2\), we prove that the Bergman metric on the regular locus of a finite ball quotient \(\Bball^N/\Gamma\), where
\(\Gamma\subset\U(N)\) is finite and fixed-point-free, is
K\"ahler--Einstein if and only if \(\Gamma\) is trivial. Consequently, if \(\Omega\) is an \(N\)-dimensional normal Stein space with isolated singularities and compact, smooth, strongly pseudoconvex boundary admitting a real-algebraic CR realization, then the Bergman metric on \(\Omega_{\Reg}\) is K\"ahler--Einstein if and only if \(\Omega\) is biholomorphic to \(\Bball^N\). This proves an algebraic version of the
Cheng--Huang--Xiao conjecture in every complex dimension \(N\ge2\).
\end{abstract}

\maketitle

\section{Introduction}\label{sec:intro}

The Bergman kernel and Bergman metric were introduced by S. Bergman
\cite{BergmanpaperonBergmankernel1,BergmanpaperonBergmankernel2} for
domains in \(\mathbb C^n\) and were later extended to complex manifolds by
S. Kobayashi \cite{GeometryboundeddomKobayashi1959}. The biholomorphic
invariance of the Bergman metric makes it a canonical K\"ahler metric in
complex geometry. Recently, there has been considerable progress in
classifying complex spaces through characterizations of their Bergman
kernels
\cite{EXX2023algebraicC2,EXX2024algebraicBergker,EGXlocalalgebraicity2025}.

Here we focus instead on classification results formulated in terms of
the Bergman metric. Every bounded strongly pseudoconvex domain carries a
natural complete K\"ahler--Einstein metric, whose existence was proved by
S.-Y. Cheng and S.-T. Yau \cite{ChengYaumetric1980}. The Cheng--Yau
metric is obtained by solving a Monge--Amp\`ere equation and is unique up
to constant scaling. Yau conjectured that the Cheng--Yau metric of a
bounded pseudoconvex domain coincides with its Bergman metric if and only
if the domain is homogeneous \cite{Semdiffgeo1982}, that is, if its
automorphism group acts transitively. This is still an open conjecture in complex geometry. 

	A  homogeneous strongly pseudoconvex bounded domain with a smooth boundary is biholomorphic to the unit ball $\mathbb{B}^n$ \cite{Wong1977}. Fu-Wong \cite{FuWong1997KEbergmanmetric}, Nemirovski-Shafikov \cite{Nemirovski_2006_ConjChengRama} and Huang-Xiao \cite{HuangXiao2021Chengalld} proved an older conjecture of Cheng \cite{ChengConjectureoriginal} which states that Yau's conjecture holds for strongly pseudoconvex domains with smooth boundary. Huang-Li \cite{HuangLi2023KEEisBergStein} extended this conjecture to Stein manifolds and proved that  \textit{the only Stein manifold with smooth and compact strongly pseudoconvex boundary for which the Bergman metric is K\"ahler-Einstein is the unit ball $\mathbb{B}^n$ up to biholomorphism}. Huang--Xiao then formulated the following generalization of Cheng's conjecture for Stein spaces with isolated normal singularities:
	
	\begin{conjecture}[\cite{HuangXiaoUniformization2020}]
		Let $\Omega$ be a normal Stein space with a compact, smooth, strongly pseudoconvex boundary. Then the Bergman metric on the regular part of $\Omega$ is K\"ahler-Einstein if and only if $\Omega$ is biholomorphic to the unit ball in a complex Euclidean space. 
	\end{conjecture}
	
The singular case is not a formal consequence of the manifold result:
the Bergman metric is defined only on the regular locus, and its behavior
near the singular set is not controlled by boundary asymptotics alone.
In this paper, we prove the following algebraic version of this
conjecture; the two-dimensional case was established in earlier joint
work of the first author \cite{GangulySinha2026ChengStein2d}.
 
\begin{theorem}
\label{thm:main-stein}
Let \(\Omega\) be a normal Stein space of dimension \(N\ge2\), with
isolated singularities and compact, smooth, strongly pseudoconvex
boundary. Assume that \(\partial\Omega\) is CR equivalent to an algebraic
CR manifold in a complex Euclidean space. Then the Bergman metric on the regular part of
\(\Omega\) is K\"ahler--Einstein if and only if \(\Omega\) is
biholomorphic to \(\Bball^N\).
\end{theorem}

The reverse implication in Theorem~\ref{thm:main-stein} follows from
the fact that the Bergman metric of \(\Bball^N\) is
K\"ahler--Einstein. For the forward implication,
\cite[Corollary~4.2]{HuangLi2023KEEisBergStein}
(see also \cite{EXX2022ChengSteinabelian,HuangCRlink2006})
shows that, under the
hypotheses of Theorem~\ref{thm:main-stein}, \(\Omega\) is biholomorphic
to a finite ball quotient \(\Bball^N/\Gamma\), where
\(\Gamma\subset\Aut(\Bball^N)\) acts fixed-point-freely on
\(\partial\Bball^N\). After conjugating the action by an automorphism
of \(\Bball^N\), we may assume that \(\Gamma\subset\U(N)\). Recall that a finite subgroup \(\Gamma\subset\U(N)\) is called
\emph{fixed-point-free} if no nonidentity element of \(\Gamma\) fixes
a point of \(\partial\Bball^N\). Equivalently,
\[
1\notin\Spec(\gamma)
\qquad
\text{for every }\gamma\in\Gamma\setminus\{I\}.
\]
Thus Theorem~\ref{thm:main-stein} reduces to the following
finite-quotient theorem.

\begin{theorem}
\label{thm:main-ball-quotient}
Let \(\Gamma\subset\U(N)\), \(N\ge2\), be a nontrivial finite fixed-point-free
subgroup.  Then the Bergman metric on the ball quotient \(\Bball^N/\Gamma\) is not
K\"ahler--Einstein.
\end{theorem}
This phenomenon is specific to dimensions at least two.  In dimension one, the
Bergman metric on the regular part of any finite disk quotient
\(\Bball^1/\Gamma\) is K\"ahler--Einstein
\cite{HuangLi2023KEEisBergStein}. Theorem~\ref{thm:main-ball-quotient} removes both the abelian hypothesis in the finite quotient theorem of
Ebenfelt--Xiao--Xu \cite{EXX2022ChengSteinabelian} and the two-dimensional
restriction in \cite{GangulySinha2026ChengStein2d}.  Although finite fixed-point-free unitary groups are classified by Wolf, with later simplifications by
Stepanov \cite{WolfSpacesofconstantcurvature2011,Stepanov2014wolfsimplified},
our proof does not use this classification.  Instead, it extracts a
valuation-theoretic obstruction directly from the polar divisor of the
complexified quotient Bergman kernel.

\subsection*{Outline of the proof}
After pulling the quotient Bergman kernel back to the ball and complexifying the real-analytic variables, the K\"ahler--Einstein condition becomes the rational identity:
\[
        F
        :=
        \det M(\phi)-(N+1)^N\phi^{N+2}
        \equiv 0
        \qquad\text{in }\C(z,w),
\]
where the potential \(\phi(z,w) =\phi_\Gamma(z,w)\) is defined by an explicit sum over the group \(\Gamma\), and \(M(\phi)\) is the bordered Monge--Amp\`ere matrix. Thus the problem is to prove that $F$ is not identically zero. The proof divides into two cases based on whether or not \(\Gamma\) contains a non-scalar matrix.

If $\Gamma$ does contain a non-scalar matrix $\eta$ then we employ an iterated valuation argument. Consider the hypersurfaces \(H_\gamma = \{1 - z^{\trans}\gamma w = 0\}\) for each group element $\gamma$, with \(H_I\) corresponding to the identity. Restricting  \(F\) along \(H_I\) yields a leading term \(G\). To further analyze \(G\) we construct a distinguished point \(p \in H_I \cap H_\eta\), which lies on no other hypersurface \(H_\gamma\). This means exactly two summands of the potential $\phi$ have a pole at $p$, and a direct computation shows that \(G\) has a genuine pole at \(p\). This proves \(G\not\equiv0\) and hence that \(F\not\equiv0\) whenever \(\Gamma\) contains a non-scalar element.

If \(\Gamma\) consists entirely of scalar multiples of $I$, then it is abelian. For abelian \(\Gamma\), Theorem~\ref{thm:main-ball-quotient} is already known by the work of Ebenfelt--Xiao--Xu \cite{EXX2022ChengSteinabelian}, which finishes the proof. 
However, for the sake of completeness, we provide a short alternative argument in the appendix. In the scalar case, the rational functions $\phi$ and $F$ depend only on \(s = z^{\trans}w\). We construct a point $b\ne0$ at which $\phi(b)=0$, and we show that the rational function $F/\phi^{N+2}$ has a genuine pole at $s=b$, thus proving that $F\not\equiv0$.

\begin{remarks}
\label{rem:intro-reformulations-extensions}

\begin{enumerate}
\normalfont
\item
Lempert's algebraic approximation theorem
\cite{LempertAlgApproxAnalGeo1995} shows that a relatively compact domain in a
reduced Stein space with only isolated singularities is biholomorphic to a
domain in an affine algebraic variety.  This
motivates the following affine-algebraic formulation of Theorem~\ref{thm:main-stein}:  Let \(V\) be an \(N\)-dimensional
affine algebraic variety, \(N\ge2\), and let \(G\Subset V\) be a relatively
compact domain such that \(\partial G\subset V_{\Reg}\).  Assume that \(V\) has
at most isolated normal singularities in \(G\), and that \(\partial G\) is smooth, strongly pseudoconvex, and algebraic.  Then the Bergman
metric on \(G_{\Reg}\) is K\"ahler--Einstein if and only if
\(G\) is biholomorphic to \(\Bball^N\).

\item
There are several recent extensions of Cheng's conjecture beyond the strongly
pseudoconvex smooth-boundary setting.  In complex dimension two,
Savale--Xiao extended Cheng's conjecture to smoothly bounded pseudoconvex
domains of finite type \cite{SavaleXiao2023KEBergmanFiniteType}.  More
recently, Hsiao--Huang--Li proved the Cheng--Yau conjecture for bounded
real-analytic pseudoconvex domains in \(\C^n\)
\cite{HuangHsiaoLi2026LocalizationChengYau}.  Sha extended the rigidity conclusion to Bergman metrics of constant
scalar curvature on bounded smoothly strongly pseudoconvex domains
\cite{sha2026rigiditycompletekahlereinsteinmetrics}.  At the end of
Section~\ref{sec:distinguished-points-and-poles}, after the proof of
Theorem~\ref{thm:main-ball-quotient}, we mention some finite-type pseudoconvex variants of
Theorem~\ref{thm:main-stein}.

\end{enumerate}
\end{remarks}

\subsection*{Acknowledgments}
The first author would like to thank Peter Ebenfelt, Ming Xiao, and
Xiaojun Huang for many enlightening discussions concerning this project. The work of the second author is partially supported by the Simons Foundation under grant number 00006698.


\section{Preliminaries}\label{sec:prelimsballquo}

We recall that a metric on a Riemannian manifold is said to be Einstein if its Ricci curvature is proportional to the metric. For a K\"ahler manifold the Einstein condition can be expressed as an equation involving the K\"ahler potential. In the case of the Bergman metric on a ball quotient $M=\Bball^N/\Gamma$, the K\"ahler--Einstein equation can be rewritten in terms of Fefferman's bordered Monge--Amp\`ere type determinant \cite{FeffermanMongeAmpere1976,FeffermanErrataMongeAmpere1976,EXX2022ChengSteinabelian}, which we now recall.

If \(u=u(z,w)\) is a rational function on $\C^N\times \C^N$ then we write $u_z= (u_{z_i})$ and $u_w =(u_{w_j})$ for the column vectors of first derivatives, $u_{zw}=(u_{z_iw_j})$ for the matrix of mixed derivatives, and we define 
\begin{equation}
\label{eq:complexified-Mu}
M(u):=
\begin{bmatrix}
u & u_w^\trans\\
u_{z} & u_{zw}
\end{bmatrix},
 \qquad F(u):=\det M(u)-(N+1)^N u^{N+2}.
\end{equation}
Next, if $\gamma$ is a unitary matrix and $\Gamma$ is a finite subgroup of $U(N)$, then we define \begin{equation}
\label{eq:complexified-phi}
h_\gamma:=1-z^{\trans}\gamma w, \qquad \phi_\Gamma(z,w):=
\sum_{\gamma\in\Gamma}
\det(\gamma) h_\gamma^{-(N+1)}, \qquad F_\Gamma :=F(\phi_\Gamma).
\end{equation}

\begin{proposition}[\cite{EXX2022ChengSteinabelian}]
\label{prop:EXX}
Let \(\Gamma\subset\U(N)\) be finite and fixed-point-free, and set
\[
\overline\Gamma:=\{\bar\gamma:\gamma\in\Gamma\}.
\]
Then the Bergman metric on the regular part of
\(\Bball^N/\Gamma\) is K\"ahler--Einstein if and only if
\begin{equation}
\label{eq:KEBnquover1}
F_{\overline\Gamma}(z,\bar z)\equiv0.
\end{equation}
\end{proposition}

The argument is carried out in \cite[Section~4.1]{EXX2022ChengSteinabelian}. Ebenfelt--Xiao--Xu use Cartan's basic map realizing the quotient and denote by \(Z\) the locus where this map is not of full rank; for a nontrivial fixed-point-free group, one has \(Z=\{0\}\). The quotient map \(\Pi:\Bball^N\longrightarrow \Bball^N/\Gamma\) is a local biholomorphism on \(\Bball^N\setminus Z\), and since the K\"ahler--Einstein condition is local, the equation may be checked after pulling the Bergman metric back to \(\Bball^N\setminus Z\). The Bergman kernel transformation formula for branched coverings, \cite[Theorem~2.3]{EXX2022ChengSteinabelian}, is then applied in \cite[Equation~(4.1)]{EXX2022ChengSteinabelian}. Substituting the explicit Bergman kernel of the ball into this transformation formula shows that, on the diagonal, the coefficient of the pulled-back Bergman kernel is \(\frac{N!}{\pi^N}\phi_{\overline\Gamma}(z,\bar z)\). Thus the pulled-back metric has local potential \(\log\phi_{\overline\Gamma}(z,\bar z)\). Writing \(\Phi_{\overline\Gamma}=\det(\partial_{z_i}\partial_{\bar z_j}\log\phi_{\overline\Gamma}(z,\bar z))\), Lemma~4.2 of \cite{EXX2022ChengSteinabelian} reduces the K\"ahler--Einstein equation to \(\Phi_{\overline\Gamma}=(N+1)^N\phi_{\overline\Gamma}(z,\bar z)\). The proposition then follows from the well-known determinant identity \(\Phi_{\overline\Gamma}=\det M(\phi_{\overline\Gamma})(z,\bar z)/\phi_{\overline\Gamma}(z,\bar z)^{N+1}\).

\medskip

The main  result that we prove in this paper is the following. 
\begin{theorem}\label{thm:phi-Gam}
If \(\Gamma\subset\U(N)\), \(N\ge2\), is a nontrivial finite fixed-point-free
subgroup, then \(F_\Gamma(z,w)\not\equiv0\).
\end{theorem}
This immediately implies Theorem \ref{thm:main-ball-quotient} as follows.

\begin{proof}[Proof of Theorem~\ref{thm:main-ball-quotient}] 
The anti-diagonal $\{(z,\bar z)\}$ is a maximal totally real subspace of $\C^N\times \C^N$. Therefore since $F_\Gamma(z,w)$ is a rational function, Theorem \ref{thm:phi-Gam} implies $F_\Gamma(z,\bar z) \not\equiv0.$  If $\Gamma$ is a nontrivial finite fixed-point-free subgroup of $U(N)$ then so is $\overline\Gamma$, and replacing $\Gamma$ by $\overline\Gamma$ we conclude \(F_{\overline \Gamma}(z,\bar z)\not\equiv0.\)  Proposition~\ref{prop:EXX} therefore implies that the Bergman
metric on \(\Bball^N/\Gamma\) is not K\"ahler--Einstein.
\end{proof}



\section{Restriction to a hypersurface}
\label{sec:restriction-hypersurface}

Our argument involves an analysis of the poles of $F_\Gamma$ from \eqref{eq:complexified-phi}, which lie along the hypersurfaces 
\begin{equation}\label{=hgam} H_\gamma := \{h_\gamma = 0\} = \{(z,w)\mid z^\trans\gamma w =1\}.
\end{equation} 

For \(\gamma\in \U(N)\), set
\[
\kappa_\gamma:=h_\gamma^{-(N+1)},\qquad
M_\gamma:=M(\kappa_\gamma),\qquad
a_\gamma:=(N+1)h_\gamma^{-1},
\]
and define matrices
\[
W_\gamma:=(\gamma w)(z^{\trans}\gamma),
\qquad
E_\gamma:=\gamma+h_\gamma^{-1}W_\gamma.
\]
Direct differentiation gives
\[
(\kappa_\gamma)_z
=a_\gamma\kappa_\gamma\gamma w,\qquad
(\kappa_\gamma)_w
=a_\gamma\kappa_\gamma\gamma^{\trans}z,\qquad
(\kappa_\gamma)_{zw}
=a_\gamma\kappa_\gamma
\bigl(a_\gamma W_\gamma+E_\gamma\bigr).
\]
Consequently,
\begin{equation}
\label{eq:M-gamma-formula}
M_\gamma=\kappa_\gamma B_\gamma,
\qquad
B_\gamma=
\begin{bmatrix}
1&a_\gamma z^{\trans}\gamma\\[2pt]
a_\gamma\gamma w&
a_\gamma\bigl(a_\gamma W_\gamma+E_\gamma\bigr)
\end{bmatrix}.
\end{equation}
The matrix \(B_\gamma\) has the block factorization
\begin{equation}
\label{eq:B-gamma-factorization}
B_\gamma=
\begin{bmatrix}
1&0\\
a_\gamma\gamma w&I
\end{bmatrix}
\begin{bmatrix}
1&0\\
0&a_\gamma E_\gamma
\end{bmatrix}
\begin{bmatrix}
1&a_\gamma z^{\trans}\gamma\\
0&I
\end{bmatrix}.
\end{equation}

We now apply this to the special case where $\gamma=I$ is the identity matrix, and to simplify notation we write 
$$ h:= h_I, \quad K =\kappa_I, \quad M_1 \coloneqq M_I= M(\kappa_I)$$

 \begin{lemma}
\label{lem:base-potential-eval}
We have $\det{M_1}=(N+1)^N h^{-(N+1)(N+2)}$ and 
\begin{equation}
\label{eq:MK-inverse}
M_1^{-1}=(N+2)h^{N+1}A_0 - h^{N+2} A, \; \text{ where }
A_0= \begin{bmatrix}
1 & 0\\
0 & 0 
\end{bmatrix}\text{ and }
A= \begin{bmatrix}
N+1 & z^{\trans}\\[4pt]
w & \frac{wz^{\trans}-I}{N+1}
\end{bmatrix}.
\end{equation}
\end{lemma}

\begin{proof}
By formulas \eqref{eq:M-gamma-formula} and \eqref{eq:B-gamma-factorization} we get $M_1= KB$ where
\begin{equation} \label{eq:Afactorization}  B=\begin{bmatrix}
 1& az^\trans\\
aw & a(awz^\trans +E) \end{bmatrix}=
\begin{bmatrix}1&0\\ aw&I\end{bmatrix}
\begin{bmatrix}1&0\\0&aE\end{bmatrix}
\begin{bmatrix}1&az^{\trans}\\0&I\end{bmatrix}.
\end{equation}
where \(a:=(N+1)h^{-1}\), \(W:=wz^{\trans}\), and
\(E:=I+h^{-1}W\). 

Since \(W^2=(z^{\trans}w)W=(1-h)W\), one has
\(E(I-W)=I\), and hence \(E^{-1}=I-W\). The rank-one determinant formula also gives
\[
\det E=1+h^{-1}z^{\trans}w=h^{-1}.
\]
Thus \(\det B=a^N\det E=(N+1)^Nh^{-(N+1)}\) and, since \(M_1\) is an
\((N+1)\times(N+1)\) matrix, we get
\[
\det M_1
=h^{-(N+1)^2}\det B
=(N+1)^Nh^{-(N+1)(N+2)}. 
\]

Inverting \eqref{eq:Afactorization} and using
\(z^{\trans}(I-W)=hz^{\trans}\) and \((I-W)w=hw\), we obtain
\[
B^{-1}=
\begin{bmatrix}
1+(N+1)z^{\trans}w & -hz^{\trans}\\
-hw & \dfrac{h}{N+1}(I-wz^{\trans})
\end{bmatrix}.
\]
Since \(z^{\trans}w=1-h\), the upper-left entry is
\(1+(N+1)z^{\trans}w=N+2-(N+1)h\). Thus \(B^{-1}=(N+2)A_0-hA\), with \(A_0, A\) as in
\eqref{eq:MK-inverse}, and consequently
\(M_1^{-1}=h^{N+1}B^{-1}=(N+2)h^{N+1}A_0 - h^{N+2} A\).
\end{proof}

We will use the previous result to study the restriction of $F_\Gamma$ to the set 
\begin{equation}
H_\Gamma:=H_I\setminus\bigcup_{\gamma \ne I}H_\gamma.
\end{equation}
First we show that $H_\Gamma$ is not empty.

\begin{lemma}
\label{lem:nonidentity-regular-on-HI} If $\Gamma$ is a finite subgroup of $U(N)$ then $H_\Gamma$ is nonempty.
\end{lemma}

\begin{proof}
Regard \(h\) as a polynomial in \(z_1\) over
\(R:=\C[z_2,\ldots,z_N,w_1,\ldots,w_N]\). Its coefficients \(-w_1\) and
\(1-\sum_{j=2}^N z_jw_j\) are relatively prime in the unique factorization
domain \(R\). Thus \(h\) is primitive and linear over the fraction field of
\(R\), so Gauss's lemma implies that \(h\) is irreducible. If
\(h_\gamma|_{H_I}\equiv0\), then the Nullstellensatz gives \(h\mid h_\gamma\).
Since both polynomials have total degree two and constant term one, this forces
\(h_\gamma=h\), hence \(z^{\trans}(\gamma-I)w\equiv0\) and \(\gamma=I\), a
contradiction. Each \(H_I\cap H_\gamma\), \(\gamma\ne I\), is therefore a proper
closed subset of the irreducible variety \(H_I\); their finite union cannot equal
\(H_I\).
\end{proof}

\begin{proposition}
\label{prop:restriction-to-HI}
Let \(\Gamma\) be a finite subgroup of \(\U(N)\), let \(A\) be as in
Lemma~\ref{lem:base-potential-eval}, and set
\[
\psi:=\sum_{\gamma\ne I}
(\det\gamma)h_\gamma^{-(N+1)},
\qquad
D:=M(\psi),
\qquad
G_\Gamma:=\tr(AD).
\]
Define
\[
F_1:=
\frac{h^{N(N+2)}}{(N+1)^N}F_\Gamma+G_\Gamma.
\]
Then \(F_1\) is regular along \(H_\Gamma\) and vanishes on
\(H_\Gamma\). In particular, if \(G_\Gamma\not\equiv0\) on
\(H_\Gamma\), then \(F_\Gamma\not\equiv0\).
\end{proposition}

\begin{proof}
Write
\[
F_\Gamma=P-Q,
\qquad
P:=\det M(\phi_\Gamma),
\qquad
Q:=(N+1)^N\phi_\Gamma^{N+2},
\]
and set
\[
\alpha:=\frac{h^{(N+1)(N+2)}}{(N+1)^N}.
\]
By Lemma~\ref{lem:base-potential-eval},
\(\alpha=(\det M_1)^{-1}\). Since
\[
M(\phi_\Gamma)=M_1+D,
\]
we have
\[
\alpha P
=
\det(I+M_1^{-1}D).
\]
Using \eqref{eq:MK-inverse}, this becomes
\[
\alpha P
=
\det\bigl(
I+(N+2)h^{N+1}A_0D-h^{N+2}AD
\bigr).
\]
The entries of \(A_0D\) and \(AD\) are regular along \(H_\Gamma\).
Every term of degree at least two in the determinant expansion is
divisible by \(h^{2N+2}\). Hence, for some rational function \(R\)
regular along \(H_\Gamma\),
\[
\alpha P
=
1+(N+2)h^{N+1}\tr(A_0D)
-h^{N+2}\tr(AD)
+h^{2N+2}R.
\]
Since
\[
\tr(A_0D)=\psi,
\qquad
\tr(AD)=G_\Gamma,
\]
we obtain
\[
\alpha P
=
1+(N+2)h^{N+1}\psi
-h^{N+2}G_\Gamma
+h^{2N+2}R.
\]

On the other hand, since
\(\phi_\Gamma=h^{-(N+1)}+\psi\), the binomial theorem gives
\[
\alpha Q
=
\bigl(1+h^{N+1}\psi\bigr)^{N+2}
=
1+(N+2)h^{N+1}\psi+h^{2N+2}S
\]
for some rational function \(S\) regular along \(H_\Gamma\).
Therefore
\[
\alpha F_\Gamma
=
-h^{N+2}G_\Gamma+h^{2N+2}(R-S).
\]
Dividing by \(h^{N+2}\), we find
\[
\frac{h^{N(N+2)}}{(N+1)^N}F_\Gamma
=
-G_\Gamma+h^N(R-S).
\]
Consequently,
\[
F_1=h^N(R-S).
\]
This expression is regular along \(H_\Gamma\) and vanishes there.

Finally, if \(F_\Gamma\equiv0\), then
\(F_1=G_\Gamma\), so the vanishing of \(F_1\) on \(H_\Gamma\)
would imply \(G_\Gamma\equiv0\) there. The last assertion follows.
\end{proof}

\section{Nonvanishing in the nonscalar case}
\label{sec:distinguished-points-and-poles}

In this section, let $\Gamma \subset \U(N), N \geq 2$ be a nontrivial finite fixed-point-free group containing a non-scalar element. Recall that \(\Gamma\) being fixed-point-free, implies \(1\notin\Spec(\gamma)\) for every \(\gamma\in\Gamma\setminus\{I\}\). Proposition~\ref{prop:restriction-to-HI} reduces the non-scalar case to proving
that \(G_\Gamma \not\equiv 0\) in \(H_{\Gamma}\). We construct a point at which
exactly one summand of \(\psi\) becomes singular on \(H_I\), and approach it
through \(H_\Gamma\) along an explicit affine curve.

\begin{lemma}
\label{lem:left-eigenvector-avoidance}
Let \(A_1,\ldots,A_k\in M_N(\C)\) be non-scalar matrices, then there exists a nonzero vector \(v \in\C^N\) which is not a left eigenvector of any \(A_j\).
\end{lemma}

\begin{proof}
For a fixed non-scalar matrix \(M\), its left eigenvectors lie in
\(\bigcup_{\lambda\in\Spec(M)}\ker((M-\lambda I)^{\trans})\), a finite union of
proper linear subspaces. A finite-dimensional vector space over an infinite field
cannot be covered by finitely many proper linear subspaces, so the same conclusion
holds for the finite collection \(A_1,\ldots,A_k\).
\end{proof}

\begin{lemma}
\label{lem:distinguished-point}
Let \(\eta\in\Gamma\) be non-scalar. Then there exist
\(p=(z,w)\in\C^N\times\C^N\) and \(u\in\C^N\) such that
\begin{equation}
\label{eq:distinguished-point-properties}
z^{\trans}w=1,
\qquad
z^{\trans}\eta w=1,
\qquad
z^{\trans}\delta w\ne1
\quad(\delta\in\Gamma\setminus\{I,\eta\}),
\end{equation}
and
\begin{equation}
\label{eq:approach-vector-properties}
z^{\trans}u=0,
\qquad
z^{\trans}\eta u\ne0.
\end{equation}
In particular, \(z^{\trans}\) and \(z^{\trans}\eta\) are linearly independent.
Moreover, for \(p_\epsilon:=(z,w+\epsilon u)\),
\begin{equation}
\label{eq:explicit-approach-curve}
p_\epsilon\in H_\Gamma
\quad\text{for every sufficiently small }\epsilon\ne0,
\qquad
p_\epsilon\to p
\quad(\epsilon\to0).
\end{equation}
\end{lemma}

\begin{proof}
Set \(S:=I-\eta\), which is invertible because \(1\notin\Spec(\eta)\). We first
verify that \(S^{-1}\) and every \(S^{-1}(I-\delta)\),
\(\delta\in\Gamma\setminus\{I,\eta\}\), are non-scalar. The assertion for
\(S^{-1}\) follows from the non-scalarity of \(\eta\). If
\(S^{-1}(I-\delta)=\mu I\), then \(\delta=(1-\mu)I+\mu\eta\). The cases
\(\mu=0,1\) give \(\delta=I,\eta\), so assume \(\mu\ne0,1\). Choose distinct
eigenvalues \(\omega_1,\omega_2\) of the non-scalar unitary matrix \(\eta\).
Since \(\delta\) is unitary, the three distinct points
\(1,\omega_1,\omega_2\) lie on both circles
\[
|\omega|=1,
\qquad
|1-\mu+\mu\omega|=1.
\]
The circles must coincide, but the second has center \(1-\mu^{-1}\); hence
\(\mu=1\), a contradiction.

By Lemma~\ref{lem:left-eigenvector-avoidance}, choose a nonzero row vector
\(v^{\trans}\) which is not a left eigenvector of any matrix just listed, and put
\(L:=\ker(v^{\trans})\). The restrictions to \(L\) of
\[
w\longmapsto v^{\trans}S^{-1}w,
\qquad
w\longmapsto v^{\trans}S^{-1}(I-\delta)w
\quad(\delta\in\Gamma\setminus\{I,\eta\})
\]
are nonzero: otherwise \(v^{\trans}A\) would vanish on \(\ker(v^{\trans})\) and
would therefore be a scalar multiple of \(v^{\trans}\), making \(v^{\trans}\) a
left eigenvector of \(A\). Since \(\dim L=N-1\ge1\), choose \(w_0\in L\)
outside the finitely many kernels of these restricted functionals, and define
\[
w:=\frac{w_0}{v^{\trans}S^{-1}w_0},
\qquad
z^{\trans}:=v^{\trans}S^{-1}.
\]
Then \(v^{\trans}w=0\), \(z^{\trans}w=1\), and
\[
z^{\trans}\eta w=1-v^{\trans}w=1,
\qquad
z^{\trans}\delta w
=1-v^{\trans}S^{-1}(I-\delta)w\ne1,
\]
which proves \eqref{eq:distinguished-point-properties}. Since
\(z^{\trans}\eta=z^{\trans}-v^{\trans}\), an equality
\(z^{\trans}\eta=c z^{\trans}\) would give
\(v^{\trans}=(1-c)v^{\trans}S^{-1}\). As \(v^{\trans}\ne0\), one has
\(c\ne1\), so \(v^{\trans}\) would be a left eigenvector of \(S^{-1}\), a
contradiction. Hence the restriction of
\(z^{\trans}\eta\) to \(\ker(z^{\trans})\) is nonzero, so one may choose
\(u\in\ker(z^{\trans})\) satisfying \(z^{\trans}\eta u\ne0\).

For \(p_\epsilon=(z,w+\epsilon u)\), one has
\(h_I(p_\epsilon)=0\) and
\(h_\eta(p_\epsilon)=-\epsilon z^{\trans}\eta u\ne0\) when \(\epsilon\ne0\).
For \(\delta\in\Gamma\setminus\{I,\eta\}\),
\(h_\delta(p_\epsilon)=h_\delta(p)-\epsilon z^{\trans}\delta u\), whose constant
term is nonzero by \eqref{eq:distinguished-point-properties}. Since there are
only finitely many such \(\delta\), all these quantities remain nonzero for every
sufficiently small \(\epsilon\ne0\), proving \eqref{eq:explicit-approach-curve}.
\end{proof}

\begin{proposition}
\label{prop:genuine-pole-at-distinguished-point}
If \(\Gamma\subset \U(N)\) is finite and fixed-point-free and contains a non-scalar element, then
\(G_\Gamma\not\equiv0\) on \(H_\Gamma\).
\end{proposition}

\begin{proof}
By definition we have
$$G_\Gamma = \tr(AD)=\sum_{\gamma \ne I} \det(\gamma)\tr(AM_\gamma)= \det(\eta)\tr(AM_\eta) + \sum_{\gamma \ne \eta, I} \det(\gamma)\tr(AM_\gamma) $$
We multiply this expression by $h_\eta^{N+3}$ and consider the limit along 
$p_\epsilon \to p=(z,w)$ as in Lemma~\ref{lem:distinguished-point}. Since the terms $AM_\gamma, \gamma \ne \eta,I$ have no poles at $p$, the factor $h_\eta^{N+3}$ causes these terms to vanish in the limit at $p$. Also by \eqref{eq:M-gamma-formula} we have  
$h_\eta^{N+3}\tr(AM_\eta) = 
h_\eta^{N+3}\kappa_\eta \tr(AB_\eta) =h_\eta^2\tr(AB_\eta). $

Now we have $$A=\begin{bmatrix}
N+1 & z^{\trans}\\[4pt]
w & \frac{wz^{\trans}-I}{N+1}
\end{bmatrix} \text{ and } h_\eta^2B_\eta = \begin{bmatrix}
0 & 0\\[4pt]
0 & (N+1)(N+2)W_\eta 
\end{bmatrix} + h_\eta C$$
where $C$ is a matrix without poles at $p$. Since \(A(p_\epsilon)\to A(p)\), it follows that
\[
\lim_{\epsilon\to0}
h_\eta(p_\epsilon)^2
\tr\bigl(A(p_\epsilon)B_\eta(p_\epsilon)\bigr)
=
(N+2)\tr\bigl((wz^{\trans}-I)W_\eta\bigr) =(N+2) \tr X.
\]
where 
$$X=(wz^\trans -I)W_\eta = (wz^\trans -I)(\eta wz^\trans\eta)=wz^\trans \eta wz^\trans\eta -\eta wz^\trans\eta.$$ 
Since $(z^\trans \eta w)=1$ at $p=(z,w)$ we get $\tr(X) = (z^\trans \eta w)^2 - z^\trans \eta^2 w =1-z^\trans \eta^2 w$, and so 

\begin{equation}
\label{eq:pole-limit}
\lim_{\epsilon\to0} h_\eta^{N+3}G_\Gamma(p_\epsilon)
=(\det\eta)\lim_{\epsilon\to0} h_\eta^{N+3}\tr(AM_\eta)=
(\det\eta)(N+2)\bigl(1-z^{\trans}\eta^2w\bigr).
\end{equation}

We claim that \(\eta^2\notin\{I,\eta\}\). Indeed, \(\eta^2=\eta\) would imply
\(\eta=I\), while \(\eta^2=I\), together with the fixed-point-free condition,
would force \(\eta=-I\), contradicting the assumption that \(\eta\) is
non-scalar. Hence, by \eqref{eq:distinguished-point-properties}, we get
\begin{equation}
\label{eq:p-not-Heta2}
z^{\trans}\eta^2w\ne1.
\end{equation}
Now \eqref{eq:pole-limit} implies that $\lim_{\epsilon\to0} h_\eta^{N+3}G_\Gamma(p_\epsilon)\ne0$, therefore $G_{\Gamma}(p_\epsilon) \ne 0$ for all sufficiently small $\epsilon>0$. Since these points are on $H_\Gamma$ the result follows. 
\end{proof}

\begin{proof}[Proof of Theorem~\ref{thm:phi-Gam}]
Suppose first that \(\Gamma\) contains a non-scalar element.
Proposition~\ref{prop:genuine-pole-at-distinguished-point} gives
\(G_\Gamma\not\equiv0\) on \(H_\Gamma\), and
Proposition~\ref{prop:restriction-to-HI} therefore yields
\(F_\Gamma\not\equiv0\).

If every element of \(\Gamma\) is scalar, then \(\Gamma\) is abelian
and \(\overline\Gamma=\Gamma\). Theorem~1.4 of
\cite{EXX2022ChengSteinabelian}, together with
Proposition~\ref{prop:EXX}, gives
\(F_\Gamma(z,\bar z)\not\equiv0\), and hence
\(F_\Gamma(z,w)\not\equiv0\). We also give a direct proof of the
scalar case in the appendix.
\end{proof}

We now mention some consequences of Theorem~\ref{thm:main-stein} where the boundary can be assumed to be of finite type instead of strongly pseudoconvex. There are different notions of finite type (for reference, please see \cite{DAngelo1982RealHypersurfaces, DAngelo1993SCV, Catlin1984BoundaryInvariants, BER1999CRgeometrybook}). In dimension two, these notions coincide (see \cite{BER1999CRgeometrybook}). The first corollary, in two dimensions, was known to Ebenfelt, Xiao, and the first author \cite{EGSX26C2}. 

\begin{corollary}
    Let $\Omega$ be a smoothly bounded, pseudoconvex, precompact domain with a finite-type boundary in a two-dimensional normal Stein space with possibly isolated singularities. Assume that $\partial \Omega$ is CR equivalent to an algebraic CR manifold in a complex Euclidean space. Then $\Omega$ has K\"ahler-Einstein Bergman metric if and only if $\Omega$ is biholomorphic to $\mathbb{B}^2$.
\end{corollary}

The preceding corollary follows from the localization theorem for the Bergman kernel \cite[Theorem~2]{EGXlocalalgebraicity2025}, Bergman kernel asymptotics of \cite{SavaleXiao2023KEBergmanFiniteType} and \cite{GangulySinha2026ChengStein2d} (or Theorem \ref{thm:main-stein}). In higher dimensions, one can prove the following result by combining
\cite{EGXlocalalgebraicity2025, HuangHsiaoLi2026LocalizationChengYau}, along with
Theorem~\ref{thm:main-stein}.

\begin{corollary}
Let $\Omega$ be a smoothly bounded, pseudoconvex, precompact domain with a finite-type boundary in an $N$-dimensional, normal Stein space with possibly isolated singularities, with $N \geq 3$. Assume that \(\partial\Omega\) is of finite
D'Angelo \(1\)-type
\cite{DAngelo1982RealHypersurfaces,DAngelo1993SCV}, and that every weakly pseudoconvex
boundary point is of \(h\)-extendible finite type in the sense of
Catlin--Yu
\cite{Catlin1984BoundaryInvariants,Yu1993GeometricAnalysis,
BoasStraubeYu1995BoundaryLimits}.  If \(\partial\Omega\) is CR equivalent to an algebraic CR manifold in a complex Euclidean space, then the
Bergman metric on the regular part of \(\Omega\) is K\"ahler--Einstein if and only if
\(\Omega\) is biholomorphic to \(\Bball^N\). 
\end{corollary}

\appendix
\section{Nonvanishing in the scalar case}
\label{app:scalar-cyclic}

As mentioned in the introduction, the scalar case is already covered by
Theorem~1.4 of \cite{EXX2022ChengSteinabelian}. We include this appendix
to show that the same rational-identity method gives a direct obstruction in
the scalar cyclic case. Thus this appendix treats the remaining case in which
every element of \(\Gamma\) is scalar; if \(\Gamma\) contains a non-scalar element,
then the argument above applies.

Throughout this appendix, assume that every element of \(\Gamma\) is scalar and
that \(\Gamma\ne\{I\}\). Let \(|\Gamma|=\ell\). Identifying \(\lambda I_N\)
with \(\lambda\), we have
\[
\Gamma=\{\lambda:\lambda^\ell=1\}.
\]
Put \(s=z^{\trans}w\). For \(\gamma=\lambda I_N\), one has
\(\det\gamma=\lambda^N\) and \(h_\gamma=1-\lambda s\). Hence the complexified
potential \eqref{eq:complexified-phi} becomes
\begin{equation}
\label{eq:scalar-f-def}
\phi_\Gamma(z,w)=f(s),
\qquad
f(s):=\sum_{\lambda^\ell=1}
\frac{\lambda^N}{(1-\lambda s)^{N+1}}.
\end{equation}

\begin{lemma}
\label{lem:radial-bordered-determinant}
Let \(u(z,w)=f(s)\), where \(s=z^{\trans}w\) and \(f\) is meromorphic. Then
\begin{equation}
\label{eq:radial-J-formula}
\det M(u)=(f')^{N-1}\bigl(ff'+s(ff''-(f')^2)\bigr).
\end{equation}
\end{lemma}

\begin{proof}
One has \(u_z=f'w\), \(u_w=f'z\), and
\(u_{zw}=L:=f'I+f''wz^{\trans}\). Applying the rank-one calculation in the proof of
Lemma~\ref{lem:base-potential-eval}, based on
\((wz^{\trans})^2=s\,wz^{\trans}\), to
\(L\), we obtain, wherever \(L\) is invertible,
\[
\det L=(f')^{N-1}(f'+sf''),
\qquad
L^{-1}=\frac1{f'}I-\frac{f''}{f'(f'+sf'')}wz^{\trans},
\qquad
z^{\trans}L^{-1}w=\frac{s}{f'+sf''}.
\]
Applying the Schur determinant formula
\cite[\S0.8.5, formula~(0.8.5.1)]{HornJohnsonMatrixAnalysis2013}
to the lower-right block \(L\) of \(M(u)\),
\[
\det M(u)=\det L\bigl(f-(f')^2z^{\trans}L^{-1}w\bigr)
=(f')^{N-1}\bigl(ff'+s(ff''-(f')^2)\bigr).
\]
Both sides are meromorphic, so the identity holds everywhere.
\end{proof}

\begin{proposition}
\label{prop:scalar-cyclic-obstruction}
If \(\Gamma\) is nontrivial and scalar, then the Bergman metric of
\(\Bball^N/\Gamma\) is not K\"ahler--Einstein.
\end{proposition}

\begin{proof}
By \eqref{eq:complexified-phi} and
Lemma~\ref{lem:radial-bordered-determinant}, \(F_\Gamma\) depends only on
\(s=z^{\trans}w\); so we write it as \(F(s)\). Then
\begin{equation}
\label{eq:scalar-F-unnormalized}
F(s)=(f')^{N-1}\bigl(ff'+s(ff''-(f')^2)\bigr)-(N+1)^Nf^{N+2}.
\end{equation}
It remains to prove that \(F\not\equiv0\).

We first choose a zero of \(f\) away from its poles. The poles are precisely the
points \(s=\lambda^{-1}\), \(\lambda^\ell=1\), each of order \(N+1\). Thus
\[
f(s)=\frac{P(s)}{(1-s^\ell)^{N+1}},
\]
where \(P\) does not vanish at any zero of \(1-s^\ell\). As \(s\to\infty\),
\[
f(s)=(-1)^{N+1}s^{-(N+1)}
\sum_{a\ge0}\binom{N+a}{a}s^{-a}
\sum_{\lambda^\ell=1}\lambda^{-1-a}.
\]
The first nonzero root-of-unity sum occurs at \(a=\ell-1\), so
\[
f(s)=(-1)^{N+1}\ell\binom{N+\ell-1}{\ell-1}s^{-(N+\ell)}
+O\bigl(s^{-(N+\ell+1)}\bigr),
\qquad
\deg P=N(\ell-1).
\]
Near \(s=0\),
\[
f(s)=\sum_{a\ge0}\binom{N+a}{a}s^a
\sum_{\lambda^\ell=1}\lambda^{N+a}.
\]
Exactly one index \(a\in\{0,\ldots,\ell-1\}\) gives a nonzero root-of-unity
sum. Hence \(P\) vanishes at \(0\) to order at most \(\ell-1\). Since
\(N\ge2\), \(\deg P=N(\ell-1)>\ell-1\), so \(P\) has a zero \(b\ne0\). By the
choice of \(P\), this \(b\) is not a pole of \(f\).

Let \(r\ge1\) be the vanishing order of \(f\) at \(b\), and write
\(f(s)=(s-b)^ru(s)\), where \(u\) is holomorphic and \(u(b)\ne0\). With
\(Y:=f'/f\), one has
\(ff'+s(ff''-(f')^2)=f^2(Y+sY')=f^2(sY)'\). Dividing
\eqref{eq:scalar-F-unnormalized} by \(f^{N+2}\) gives
\[
\frac{F}{f^{N+2}}=\frac{Y^{N-1}(sY)'}{f}-(N+1)^N.
\]
Since
\[
Y=\frac{r}{s-b}+\frac{u'}u,
\qquad
(sY)'=-\frac{rb}{(s-b)^2}+O(1),
\]
we obtain
\[
\frac{Y^{N-1}(sY)'}{f}
=-\frac{br^N}{u(b)}(s-b)^{-(N+r+1)}
+\text{lower-order pole terms}.
\]
The leading coefficient is nonzero because \(b\ne0\). Therefore
\(F/f^{N+2}\) has a genuine pole at \(b\), and hence \(F\not\equiv0\). The
conclusion follows from Proposition~\ref{prop:EXX}.
\end{proof}

\bibliographystyle{alpha}
\bibliography{Chengref}

\end{document}